\documentclass[11pt]{article}

\usepackage{amsfonts,amssymb}

\def\Sym{S}

\def\Orb{{\mathcal O}}

\makeatletter
\def\revddots{\mathinner{\mkern1mu\raise\p@\vbox{\kern7\p@\hbox{.}}\mkern2mu\raise4\p@\hbox{.}\mkern2mu\raise7\p@\hbox{.}\mkern1mu}}
\makeatother

\newcommand\eqref[1]{(\ref{#1})}

\newenvironment{proof}{{\noindent\bf Proof.}}{\hfill $\square$}

\def\longto{\longrightarrow}

\def\Part{{\mathcal P}}\def\lr{{\mathcal LR}}

\def\PP{{\mathbb P}}
\def\QQ{{\mathbb Q}}\def\ZZ{{\mathbb Z}}
\def\CC{{\mathbb C}}

\def\Fl{{\mathcal Fl}}\def\Face{{\mathcal F}}
\def\Gr{{\mathbb G}}

\def\LR{\rm LR}

\def\Li{{\mathcal{L}}}

\def\Om{\Omega}

\def\GL{{\rm GL}}

\def\hnu{{\hat\nu}}
\def\hG{{\hat G}} \def\hB{{\hat B}}\def\hT{{\hat T}}\def\hU{{\hat U}} 
\def\hW{{\hat W}}\def\hw{{\hat w}}
\def\hP{{\hat P}}\def\hbeta{{\hat \beta}}

\newtheorem{theo}{Theorem}
\newtheorem{coro}{Corollary}

\newenvironment{defin}{Definition}{}
\newenvironment{remark}{Remark}{}

\begin{document}
\title{Reductions for branching coefficients}
\author{N. Ressayre}

\maketitle
\begin{abstract}
Let $G$ be a connected reductive subgroup of a complex connected reductive group $\hat{G}$.
Fix maximal tori and Borel subgroups of $G$ and $\hat G$.
Consider the cone $\lr(G,\hG)$ generated by the pairs $(\nu,\hnu)$ of  dominant characters 
such that $V_\nu^*$  is a sub-$G$-module of $V_{\hnu}$.
It is known that $\lr(G,\hG)$ is a closed convex polyhedral cone.
In this work, we show that every regular face of $\lr(G,\hG)$ gives
rise to a {\it reduction rule} for multiplicities.
More precisely, for $(\nu,\hnu)$ on such a face, the
multiplicity of $V_\nu^*$  in $V_{\hnu}$ is proved to be equal to a similar
multiplicity for representations of Levi subgroups of $G$ and $\hat G$. 
This generalizes, by different methods, results obtained by Brion, Derksen-Weyman, Roth\dots 
\end{abstract}

\section{Introduction}

Let $G$ be a connected reductive subgroup of a complex connected reductive group $\hat{G}$.
The branching problem consists in 
\begin{center}
  decomposing irreducible representations of $\hat G$ as sum of
  irreducible $G$-modules.
\end{center}

Fix maximal tori $T\subset\hT$ and Borel subgroups $B\supset T$ and
$\hB\supset \hT$ of $G$ and $\hG$.
Let $X(T)$ denote the group of characters of $T$ and let $X(T)^+$
denote the set of dominant characters.
For $\nu\in X(T)^+$, $V_\nu$ denotes the
irreducible representation of highest weight $\nu$.
Similarly we use notation $X(\hat T)$,  $X(\hat T)^+$,  $V_{\hat\nu}$
relatively to $\hat G$. 
For any $G$-module $V$, the subspace of
 $G$-fixed vectors is denoted by $V^G$.
For $\nu\in X(T)^+$ and $\hat\nu\in X(\hat T)^+$, set
\begin{eqnarray}
  \label{eq:defc}
  c_{\nu\,\hat\nu}(G,\hat G)=\dim(V_\nu\otimes V_{\hat\nu})^G.
\end{eqnarray}
Sometimes we simply write $c_{\nu\,\hat\nu}$ for $c_{\nu\,\hat\nu}(G,\hat G)$.
Let $V_\nu^*$ denote the dual representation of $V_\nu$.
The branching problem is equivalent to the knowledge of these coefficients
since 
\begin{eqnarray}
  \label{eq:2}
  V_{\hat\nu}=\sum_{\nu\in X(T)^+}c_{\nu\,\hat\nu} V_\nu^*.
\end{eqnarray}
The set $\LR(G,\hG)$ of pairs $(\nu,\hnu)\in X(T)^+\times X(\hat T)^+$
such that $c_{\nu\,\hat\nu}\neq 0$ is known to be 
is a finitely generated subsemigroup of the free abelian group
$X(T)\times X(\hat T)$ (see \cite{Elash}).
Consider the convex cone $\lr(G,\hG)$ generated in $(X(T)\times X(\hat T))\otimes\QQ$
by $\LR(G,\hG)$. It is a closed convex polyhedral cone in $(X(T)\times X(\hat T))\otimes\QQ$.

\bigskip
Let  $\Face$ be a face of $\lr(G,\hG)$. Assume that $\Face$ is
regular, that is that it contains pairs  $(\nu,\hnu)$ of
regular dominant weights. 
Let $\hW$ be the Weyl group of $\hG$ and $\hT$.
If $S$ is a torus in $G$ and $H$ is a subgroup of $G$ containing $S$,
$H^S$ denotes the centralizer of $S$ in $H$.
By \cite{GITEigen2}, the regular face $\Face$ corresponds to a pair
$(S,\hw)$ where $S$ is a subtorus of $T$ and $\hw\in\hat W$ such that
\begin{eqnarray}
  \label{eq:4}
  \hG^S\cap \hw \hB\hw^{-1}=\hB^S,
\end{eqnarray}
and the span of $\Face$ is the set of pairs $(\nu,\hnu)\in (X(T)\times
X(\hat T))\otimes\QQ$ such that 
\begin{eqnarray}
  \label{eq:5}
  \nu_{|S}+\hw\hnu_{|S}=0\in X(S)\otimes \QQ.
\end{eqnarray}

\begin{theo}
\label{th:mainintro}
Let $(\nu,\hnu)\in X(T)^+\times X(\hT)^+$ be a pair of dominant
weights.
Assume that $(\nu,\hnu)$ belongs to the span of $\Face$
(equivalently that it satisfies condition \eqref{eq:5}).
Then 
$$
c_{\nu\,\hat\nu}(G,\hat G)=c_{\nu\,\hw\hat\nu}(G^S,\hat G^S).
$$ 
\end{theo}

Theorem \ref{th:mainintro} is the algebraic conterpart of the 
 geometric Theorem~\ref{th:geom} below.
Let $X=G/P\times \hG/\hP$ be a flag manifold for the group $G\times \hG$.
Let $\lambda$ be a one-parameter subgroup of $G$ and $C$ be an
irreducible component of the fixed point set $X^\lambda$ of $\lambda$
in $X$. 
Let $G^\lambda$ be the centralizer of the image of $\lambda$ in $G$. 
We assume that $(C,\lambda)$ is a {\it (well) covering pair} in the
sense of \cite[Definition 3.2.2]{GITEigen} (see also Definition \ref{def:wc}
below).

\begin{theo}\label{th:geom}
Let $\Li$ be a $G$-linearized line bundle on $X$ generated by its
global sections such that $\lambda$
acts trivially on the restriction $\Li_{|C}$.
  Then the restriction map induces an isomorphism 
$$
H^0(X,\Li)^G\longto H^0(C,\Li_{|C})^{G^\lambda},
$$
between the spaces of invariant sections of $\Li$ and $\Li_{|C}$.
\end{theo}

Several particular cases of Theorems \ref{th:mainintro} and
\ref{th:geom} was known before. 
If $G=T$ is a maximal torus of $G=\GL_n$, our theorem is equivalent to 
\cite[Theorem 5.8]{KTT:factorKostka}.
If $\hG=G\times G$ (or more generally
$\hG=G^s$ for some integer $s\geq 2$) and $G$ is diagonally
embedded in  $\hG$ then $c_{\nu\,\hnu}(G,\hG)$
(resp. $c_{\nu\,\hw\hat\nu}(G^S,\hat G^S)$)
are tensor product multiplicities for the group $G$ (resp. $G^S$).
This case was recently proved independently by Derksen and Weyman in \cite[Theorem 7.4]{DW:comb}
and King, Tollu and Toumazet in \cite[Theorem 1.4]{KTT:factorLR}
if $G=GL_n$ and for any reductive group by Roth in \cite{Roth:red}. 
If $\nu$ is regular then Theorem \ref{th:geom} can be obtained
applying \cite[Theorem 3]{Br:genface} and \cite{GITEigen}.
Similar reductions can be found in \cite{Br:Foulkes,man,Montagard:plethysm}.  

Note that our proof is new and uses strongly the normality of the
Schubert varieties. For example, in Roth's proof (which may be the
closest from our) the normality of Schubert varieties play no role.
In \cite{DW:comb}, the case $GL_n\subset\GL_n\times \GL_n$ is obtained
as a consequence of a more general result on quivers. 
Derksen-Weyman's theorem on quivers can be proved by the method used here.

In Section \ref{sec:exple}, Theorem \ref{th:geom} is applied to recover
known results in representation theory.

\bigskip
{\bf Acknowledgment.}
This work was motivated by  Roth's paper \cite{Roth:red}. I want to
thank Mike Roth for stimulating discussions on it.

\section{Proof of Theorem \ref{th:geom}}
\label{sec:pf2}

Consider the variety $X=G/P\times \hG/\hP$ endowed with the
diagonal $G$-action: $g'.(gP/P,\,\hat g\hat P/\hP)=(g'gP/P,\,g'\hat
g\hat P/\hP)$.

Let $\lambda$ be a one-parameter subgroup of $G$. 
Consider the centralizer $G^\lambda$ of $\lambda$ in $G$ and
the parabolic subgroup (see \cite{GIT})
$$
P(\lambda)=\left \{
g\in G \::\: \lim_{t\to 0}\lambda(t).g.\lambda(t)^{-1} 
\mbox{  exists in } G \right \}.
$$

Let $C$ be an irreducible component of the fixed point set $X^\lambda$
of $\lambda$ in $X$.
Set
\begin{eqnarray}
  \label{eq:tildeC+}
C^+:=\{x\in X\::\: \lim_{t\to 0}\lambda(t)x {\rm\ belongs\ to\ }C\}.
  \end{eqnarray}
Note that $C^+$ is $P(\lambda)$-stable and  locally closed in $X$.
Consider the  subvariety $Y$ of $G/P(\lambda)\times X$ defined by
$$
Y=\{(gP(\lambda)/P(\lambda),x)\;:\;g^{-1}x\in C^+\}.
$$
The morphism $\pi\,:\,G\times C^+\longto Y,\, (g,x)\longmapsto (gP(\lambda)/P(\lambda),\,gx)$
identifies $Y$ with the quotient of $G\times C^+$ by the action of $P(\lambda)$
given by $p.(g,x)=(gp^{-1},px)$. 
The variety $Y$ is denoted by $G\times_{P(\lambda)} C^+$. Set $[g:x]=\pi(g,x)$ and
consider the $G$-equivariant map
$$
\begin{array}{cccc}
\eta\ :&G\times_{{P(\lambda)}}C^+&\longto& X\\
       &[g:x]&\longmapsto&g.x.
\end{array}
$$ 

Recall from \cite{GITEigen} the notion of well covering pairs.

\begin{defin}
\label{def:wc}
The pair $(C,\lambda)$ is said to be {\it covering} if
$\eta$ is birational.
The pair $(C,\lambda)$ is said to be {\it well covering} if  there exists  a $P(\lambda)$-stable open subset $\Om$ of 
$C^+$ intersecting $C$  such that $\eta$ induces an isomorphism from 
$G\times_{{P(\lambda)}}\Om$ onto an open subset of $X$.
\end{defin}

\begin{proof}[of Theorem \ref{th:geom}]
Consider the closure $\overline{C^+}$ of $C^+$ in $X$.
Since $(C,\lambda)$ is covering the map 
$$
\begin{array}{cccc}
\overline{\eta}\,:&
G\times_{P(\lambda)}\overline{C^+}&\longto &X\\
&[g:x]&\longmapsto&gx
\end{array}
$$
is proper and birational. Hence it induces a $G$-equivariant isomorphism
$$
H^0(X,\Li)\simeq H^0(G\times_{P(\lambda)} \overline{C^+},\overline{\eta}^*(\Li)).
$$
In particular
$$
H^0(X,\Li)^G\simeq H^0(G\times_{P(\lambda)} \overline{C^+},\overline{\eta}^*(\Li))^G.
$$
We embed  $\overline{C^+}$ in $G\times_{P(\lambda)}\overline{C^+}$, by
$x\longmapsto [e:x]$. 
Note that the composition of the immersion of $\overline{C^+}$ 
in $G\times_{P(\lambda)}\overline{C^+}$ with $\overline{\eta}$ is the
inclusion map from   $\overline{C^+}$ to $X$.
 In particular $\overline{\eta}^*(\Li)_{|\overline{C^+}}=\Li _{|\overline{C^+}}$
and the restriction induces the following isomorphism
(see for example \cite[Lemma~4]{GITEigen}):
$$
H^0(G\times_{P(\lambda)} \overline{C^+},\overline{\eta}^*(\Li))^G\simeq H^0(\overline{C^+},\Li _{|\overline{C^+}})^{P(\lambda)}.
$$
Since once more, the composition of the immersion of $\overline{C^+}$ 
in $G\times_{P(\lambda)}\overline{C^+}$ with $\overline{\eta}$ is the immersion
of $\overline{C^+}$ in $X$,  we just proved that the restriction induces
the following isomorphism

\begin{eqnarray}
  \label{eq:CP}
  H^0(X,\Li)^G\simeq H^0(\overline{C^+},\Li _{|\overline{C^+}})^{P(\lambda)}.
\end{eqnarray}

Since $\lambda$ acts trivially on $\Li_{|C}$, 
\cite[Lemma~5]{GITEigen} proves that 
the restriction maop induces the following isomorphism

\begin{eqnarray}
  \label{eq:CL}
H^0(C^+,\Li _{|C^+})^{P(\lambda)}\simeq H^0(C,\Li _{|C})^{G^\lambda}.
\end{eqnarray}

By isomorphisms \eqref{eq:CP} and \eqref{eq:CL}, it remains to prove that the
restriction induces the following isomorphism 
$$
H^0(\overline{C^+},\Li _{|\overline{C^+}})^{P(\lambda)}\simeq H^0(C^+,\Li _{|C^+})^{P(\lambda)};
$$
that is, that any regular $P(\lambda)$-invariant section $\sigma$ of $\Li$ on $C^+$ 
extends to $\overline{C^+}$.

\bigskip
Note that $\lambda$ is also a one-parameter subgroup of $\hG$ and
that $\hP(\lambda)$ is defined. 
Fix a maximal torus $T$ of $G$ containing the image of
$\lambda$ and a maximal torus $\hT$ of $\hG$ containing $T$.
Note that $P$ and $\hP$ have not been fixed up to now; we have only
considered the $G\times \hG$-variety $X$. In other words, we can change 
$P$ and $\hP$ by conjugated subgroups. 
Fix a $T\times\hT$-fixed point $x_0$ in $C$, and 
 denote by $P\times \hP$ its stabilizer in $G\times \hG$.
 Hence $x_0=(P/P,\hP/\hP)$.

It is well known that $C^+=P(\lambda)P/P\times \hP(\lambda)\hP/\hP$. 
In particular
$\overline{C^+}$ is a product of Schubert varieties and is normal. 
Hence it is sufficient to proved that $\sigma$ has no pole.
Since $\sigma$ is regular on $C^+$, it remains  to prove that  
$\sigma$ has no pole along any codimension one irreducible component $D$
 of $\overline{C^+}-C^+$. 
We are going to compute the order of the pole of $\sigma$ along $D$ by
a quite explicit computation in a neighborhood of $D$ in $\overline{C^+}$.

\bigskip
If $\beta$ is a root of $(T,G)$,  $s_\beta$ denotes the
associated reflection in the Weyl group.
The divisor $D$ is the closure of 
$P(\lambda).s_\beta P/P\times \hP(\lambda)\hP/\hP$ 
for some  root $\beta$ or of 
$P(\lambda)P/P\times \hP(\lambda) s_\hbeta\hP/\hP$ 
for some root $\hbeta$.
Consider the first case. The second one works similarly.
 
Set $y =(s_\beta P/P,\hP/\hP)$; it is a point in $D$.
Consider the unipotent radical $U^-$  of the parabolic subgroup of
$G$ containing $T$ and opposite to $P$. Similarly define $\hU^-$.
Consider the groups $U_y=P(\lambda)\cap s_\beta U^-s_\beta $ and
$\hat U_y=\hP(\lambda)\cap \hU^-$.
Let $\delta$ be the $T$-stable line in $G/P$ containing $P/P$ and
$s_\beta P/P$.
Consider the map 
$$
\begin{array}{cccl}
  \theta\,:&U_y\times \hat U_y\times (\delta-\{P/P\})&\longto &X\\
&(u,\hat u,x)&\longmapsto&(ux,\hat u\hP/\hP).
\end{array}
$$
The map $\theta$ is an immersion 
and its image $\Omega$ is open in $\overline{C^+}$.
Since $\Omega$ intersects $D$, it is sufficient to prove that $\sigma$ extends on  $\Omega$.
Equivalently, we are going to prove that $\theta^*(\sigma)$ extends to
a regular section of $\theta^*(\Li)$.

The torus $T$ acts on $U_y\times \hat U_y\times (\delta-\{P/P\})$ by
$t.(u,\hat u,x)=(tut^{-1},t\hat u t^{-1},tx)$. This action makes
$\theta$ equivariant.
The curve $(\delta-\{P/P\})$ is isomorphic to $\CC$. 
The group $U_y$ is unipotent and so isomorphic to its Lie algebra.
It follows that $U_y\times \hat U_y\times (\delta-\{P/P\})$ is isomorphic as a 
$T$-variety to an affine space $V$ with linear action of $T$.

Fix root (for the action of $T\times\hT$)  coordinates $\xi_i$ on the Lie algebra of $U_y\times \hU_y$.
Fix a $T$-equivariant coordinate $\zeta$ on $\delta-\{P/P\}$. 
Then $(\xi_i,\zeta)$ are coordinates on $V$.
 Let $(a_i,a)$ be the opposite of the 
weights of the variables for the action of $\lambda$.
The weights of $T$ corresponding to the part $U_y$ are roots of
$P(\lambda)$ and the weights of $\hT$ corresponding to the part $\hU_y$ are roots of
$\hP(\lambda)$.  The weight of the action of $T$ on $T_{s_\beta
  P/P}\delta$ is a root of $G$ but not of $P(\lambda)$.
 Then we have
\begin{eqnarray}
  \label{eq:sgn1}
  a_i\geq 0\ {\rm and\ }a<0.
\end{eqnarray}

Note that $(\iota\circ\theta)^{-1}(D)$ is the divisor $(\zeta=0)$ on $V$.
 
Consider now, the $\CC^*$-linearized line bundle $\theta^*(\Li)$ on $V$.
It is trivial as a line bundle (the Picard group of $V$ is trivial) and so, 
it is isomorphic to $V\times \CC$ linearized by 
$$
t.(v,\tau)=(\lambda(t)v,t^\mu\tau)\ \ \forall t\in\CC^*, 
$$
for some integer $\mu$.

\bigskip
We first admit that  
\begin{eqnarray}
  \label{eq:sgnmu}
  \mu\leq 0
\end{eqnarray}
 and we end the proof.
The section $\theta^*(\sigma)$ corresponds to a polynomial in the variables
$\xi_i,\zeta$ and $\zeta^{-1}$; that is, a linear combination of  
monomials $m=\prod_i \xi_i^{j_i}.\zeta^j$ for some $j_i\in\ZZ_{\geq 0}$ and
$j\in\ZZ$. 
The opposite of the weight of $m$ for the action of $\CC^*$ is
$
\sum_i j_ia_j+ja.
$
The fact that $\sigma$ is $\CC^*$-invariant  implies that the monomials 
occurring in the expression of $(\iota\circ\theta)^*(\sigma)$ satisfy
$$
\sum_i j_ia_j+ja=\mu.
$$
Hence
$$
j=\frac {-1} a (\sum_i j_ia_i-\mu).
$$
Now, inequalities  \eqref{eq:sgn1} and \eqref{eq:sgnmu} imply that $j\geq 0$.
In particular $(\iota\circ\theta)^*(\sigma)$ extends to a regular function 
on $V$. 
It follows that $\sigma$ has no pole along $D$.

\bigskip
It remains to prove inequality \eqref{eq:sgnmu}.
Consider the restriction of $\Li$ to $\delta$.
Note that $\delta$ is isomorphic to $\PP^1$ and $\Li_{|\delta}$
is isomorphic to $\Orb(d)$ as a line bundle for some integer $d$.
Since $\Li$ is semiample, $d$ is nonnegative.
The group $\CC^*$ acts on $T_{P/P}\delta$ by the weight $-a$ and on
$T_y\delta$ be the weight $a$.
By assumption, the group $\CC^*$ acts trivially on the fiber $\Li_{x_0}$ 
(recall that $x_0$ belongs to $C$). It acts on the fiber $\Li_y$ by the
weight $\mu$.
Now, the theory of $\PP^1$ implies that:
$$
d=\frac{\mu-0}{a}.
$$
But, $d\geq 0$ and $a<0$. It follows that $\mu\leq 0$.  
\end{proof}

 \section{Proof of Theorem \ref{th:mainintro}}

Let $T$, $B$, $\hT$ and $\hB$ be like in the introduction.
For any character $\nu$ of $B$,  $\Li_{\nu}$  denotes the $G$-linearized line
bundle on $G/B$ such that $B$ acts on the fiber in $\Li_\nu$ over
$B/B$ with the weight $-\nu$.
By Borel-Weil's theorem, the line bundle $\Li_\nu$ is generated by its global sections if and
only if $\nu$ is dominant and in this case $H^0(G/B,\Li_\nu)$ is
isomorphic to the dual $V_\nu^*(G)$ of the  simple $G$-module $V_\nu(G)$ with highest weight $\nu$.

Consider the complete flag variety $X=G/B\times \hG/\hB$ of the group
$G\times \hG$.
Let $\nu$ and $\hnu$ be like in Theorem \ref{th:mainintro}.
Let $\Li$ be the exterior product on $X$ of $\Li_\nu$ and $\Li_\hnu$.
By Borel-Weil's theorem (applied to the group $G\times\hG$), we have
$$
V_\nu(G)^*\otimes V_\hnu^*(\hG)=H^0(X,\Li).
$$
In particular $c_{\nu\,\hnu}(G,\hG)$ is the dimension of
$H^0(X,\Li)^G$.

Set $C=G^SB/B\times \hG^S\hw\hB/\hB$. 
By \cite{GITEigen2}, there exists a one-parameter subgroup $\lambda$ of
$S$ such that $(C,\lambda)$ is well covering and $G^S=G^\lambda$.
 Moreover, assumption \eqref{eq:5} implies  that   $\lambda$ acts
trivially on $\Li_{|C}$. Hence Theorem  \ref{th:geom} implies that
$$
H^0(X,\Li)^G\simeq H^0(C,\Li_{|C})^{G^S}.
$$
However $C$ is isomorphic to the complete flag manifold of the group
$G^S\times \hG^S$. 
By condition \eqref{eq:4}, $\Li_{|C}$ is the line bundle
$\Li_\nu\otimes\Li_{\hw\hnu}$. Hence Borel-Weil's theorem implies that
$H^0(C,\Li_{|C})$ is isomorphic to $V_\nu^*(G^S)\otimes
V_{\hw\hnu}^*(\hG^S)$.
In particular $c_{\nu\,\hw\hat\nu}(G^S,\hat G^S)$ is the dimension of 
$H^0(C,\Li_{|C})^{G^S}$. The theorem is proved.

\section{Examples}
\label{sec:exple}

\subsection{Tensor product decomposition}

In this subsection, we consider the case when $\hG=G\times G$ and $G$
is diagonally embedded in $\hG$.
Assume that $\hB=B\times B$ and $\hT=T\times T$.
Then a dominant weight $\hnu$ of $\hT$ is a pair $(\lambda,\mu)$ of
dominant weights of $T$ and $V_\hnu(G\times G)=V_\lambda(G)\otimes V_\mu(G)$.
For short, we  denote by $c_{\lambda\,\mu\,\nu}(G)$ the coefficient
$c_{\nu\,\hnu}(G,\hG)$.
Then
\begin{eqnarray}
  \label{eq:11}
  V_\lambda(G)\otimes V_\mu(G)=\sum_{\nu}c_{\lambda\,\mu\,\nu}(G)\;V_\nu^*(G),
\end{eqnarray}
and $c_{\lambda\,\mu\,\nu}(G)$  is a tensor product multiplicity for $G$.
With the  notations of Theorem \ref{th:mainintro}, we have $\hG^S=G^S\times
G^S$. In particular the coefficient $c_{\nu\,\hw\hat\nu}(G^S,\hat
G^S)$ is a tensor product multiplicity for the Levi subgroup $G^S$ of
$G$. Hence Theorem \ref{th:mainintro} implies to the
main result of \cite{Roth:red}. 

\bigskip
Consider the case when $G=\GL_n(\CC)$, $T$ consists in diagonal
matrices and $B$ in upper triangular matrices.
Then a dominant weight $\lambda$ is a nonincreasing sequence
$(\lambda_1,\,\cdots,\lambda_n)$ of $n$ integers and
$c_{\lambda\,\mu\,\nu}(G)$ is a Littlewood-Richardson coefficient
denoted by $c_{\lambda\,\mu\,\nu}^n$.

Notations are useful  to describe $LR(G,\hG)$.
Let $\Gr(r,n)$ be the Grassmann variety of $r$-dimensional subspaces
of  $\CC^n$.
Let $F_\bullet$: $\{0\}=F_0\subset F_1\subset 
F_2\subset\cdots\subset F_{n}=\CC^n$ be the standard flag   of $\CC^n$.
Let $\Part(r,n)$ denote the set of subsets of $\{1,\cdots,n\}$ with $r$ elements. 
Let $I=\{i_1<\cdots<i_r\}\in\Part(r,n)$. 
The Schubert variety $\Omega_I(F_\bullet)$ in $\Gr(r,n)$  is defined by
$$
\Omega_I(F_\bullet)=\{L\in\Gr(r,n)\,:\,
\dim(L\cap F_{i_j})\geq j {\rm\ for\ }1\leq j\leq r\}.
$$
The Poincar\'e dual of the homology class of $\Omega_I(F_\bullet)$  is denoted by $\sigma_I$.
The classes $\sigma_I$ form a $\ZZ$-basis for the cohomology ring of $\Gr(r,n)$. 
The class associated to $[1;r]$ is the class of the point; it is
denoted by $[{\rm pt}]$.

By \cite{Kly:stable} , \cite{KT:saturation} and finally \cite{Belk:P1},
we have the following statement.

\begin{theo}\label{th:Belk}
  Let $(\lambda,\mu,\nu)$ be a triple of nonincreasing sequences of $n$ integers.
Then $c_{\lambda\,\mu\,\nu}^n\neq 0$ if and only if 
\begin{eqnarray}
\label{eq:traceBelk}
  \sum_i\lambda_i+\sum_j\mu_j+\sum_k\nu_k=0
\end{eqnarray}
 and
\begin{eqnarray}
  \label{eq:ineqIJKBelk}
\sum_{i\in I}\lambda_i+\sum_{j\in J}\mu_j+\sum_{k\in K}\nu_k\leq 0,
\end{eqnarray}
for any $r=1,\cdots,n-1$, for any $(I,J,K)\in \Part(r,n)^3$ such that 
\begin{eqnarray}
  \label{eq:cohomcond}
\sigma_I.\sigma_J.\sigma_K=[{\rm pt}]\in {\rm H}^*(\Gr(r,n),\ZZ).
\end{eqnarray}
\end{theo}
 
Knutson, Tao and Woodward proved in 
\cite{KTW} that this statement is optimal in the following sense.

\begin{theo}\label{th:KTW}
  In Theorem~\ref{th:Belk}, no inequality can be omitted.
\end{theo}

In other words, each inequality \eqref{eq:ineqIJKBelk} corresponds to a
regular face $\Face_{IJK}$  of the cone $\lr(G,\hG)$.
For $I=\{i_1<\cdots<i_r\}\in\Part (r,n)$ and $\lambda$ a sequence of
$n$ integers,  set 
$
\lambda_I=(\lambda_{i_1},\cdots,\lambda_{i_r})\in\ZZ^r.
$
Denote by $I^c\in\Part(n-r,n)$ the complement of $I$ in $\{1,\cdots,n\}$.
It is easy to
check that Theorem \ref{th:mainintro} gives in this case the following
statement.

\begin{theo}\label{th:prod}
  Let $(\lambda,\mu,\nu)$ be a triple of nonincreasing sequences of $n$ integers. Let
  $(I,J,K)\in \Part(r,n)$ such that 
  \begin{eqnarray}
    \label{eq:17}
\sigma_I.\sigma_J.\sigma_K=[{\rm pt}].
  \end{eqnarray}
If 
\begin{eqnarray}
  \label{eq:18}
  \sum_{i\in I}\lambda_i+\sum_{j\in J}\mu_j+\sum_{k\in K}\nu_k=0
\end{eqnarray}
 then
 \begin{eqnarray}
   \label{eq:12}
c^n_{\lambda\,\mu\,\nu}=c^r_{\lambda_I\,\mu_J\,\nu_K}\;.\; 
c^{n-r}_{\lambda_{I^c}\,\mu_{J^c}\,\nu_{K^c}}.
 \end{eqnarray}
\end{theo}

Theorem \ref{th:prod} has been proved independently in
\cite{KTT:factorLR} and \cite{DW:comb}.
Note that if equation \eqref{eq:18} does not hold then $c^r_{\lambda_I\,\mu_J\,\nu_K}=0$. 

It is known that Theorem \ref{th:Belk} also holds if condition
\eqref{eq:cohomcond} is replaced by 
\begin{eqnarray}
  \label{eq:38}
  \sigma_I.\sigma_J.\sigma_K=d[{\rm pt}]\in {\rm H}^*(\Gr(r,n),\ZZ),
\end{eqnarray}
for some positive integer $d$.
The following example shows that condition \eqref{eq:17} cannot be
replaced by condition
\eqref{eq:38} in Theorem \ref{th:prod}.

\bigskip\noindent{\bf Example.}
  Here $n=6,\, r=3$ and $I=J=K=\{1,\,3,\,5\}$.
Then $\sigma_I.\sigma_J.\sigma_K=2[{\rm pt}]$ and for any 
$(\lambda,\mu,\nu)$ in $\LR(G,\hG)$, the inequality 
$\sum_{i\in I}\lambda_i+\sum_{j\in J}\mu_j+\sum_{k\in K}\nu_k\leq 0$
holds.
Consider   $\lambda=\mu=\nu=(1\,1\,0\,0\,-1\,-1)$.
Then  $c^n_{\lambda\,\mu\,\nu}=3$. Hence $(\lambda,\mu,\nu)$ belongs
to  $\LR(G,\hG)$.
Moreover $\lambda_I=\mu_J=\nu_K=\lambda_{I^c}=\mu_{J^c}=\nu_{K^c}=
(1\,0\,-1)$
and 
$\sum_{i\in I}\lambda_i+\sum_{j\in J}\mu_j+\sum_{k\in K}\nu_k=0$.
But  $c^r_{\lambda_I\,\mu_J\,\nu_K}= 
c^{n-r}_{\lambda_{I^c}\,\mu_{J^c}\,\nu_{K^c}}=2$ and 
$c^r_{\lambda_I\,\mu_J\,\nu_K}.
c^{n-r}_{\lambda_{I^c}\,\mu_{J^c}\,\nu_{K^c}}=4\neq 3=c^n_{\lambda\,\mu\,\nu}$.

\bigskip
\begin{remark}
  With notation of Section~\ref{sec:pf2}, if $\eta$ is dominant, the map
$$
H^0(X,\Li)^G\longto H^0(C,\Li)^{G^\lambda}
$$
is injective.
When applied to $X=\Fl(n)^3$ this observation showh that if
$\sigma_I.\sigma_J.\sigma_K\neq 0$ then equality~\ref{eq:18} implis
that $c^n_{\lambda\,\mu\,\nu}\leq c^r_{\lambda_I\,\mu_J\,\nu_K}\;.\; 
c^{n-r}_{\lambda_{I^c}\,\mu_{J^c}\,\nu_{K^c}}$ according to the example.
\end{remark}

\bigskip
Note that Knutson and  Purbhoo proved in \cite{KnPu:product} some equalities
\eqref{eq:12} with assumptions different from those of Theorem \ref{th:prod}.
 
\subsection{Kronecker coefficients}

Let $\alpha=(\alpha_1\geq\alpha_2\geq\dots)$ be a partition. 
Denote by $l(\alpha)$ the number of nonzero parts of $\alpha$.
Set $|\alpha|=\sum_i\alpha_i$,  $\alpha$ is called a partition
of $|\alpha|$.
Consider the symmetric group $\Sym_n$ acting on $n$ letters.
The irreducible representations of $\Sym_n$ are parametrized by the
partitions of $n$, let $[\alpha]$ denote the representation
corresponding to $\alpha$.
The Kronecker coefficients $k_{\alpha\,\beta\,\gamma}$, depending on
three partitions $\alpha,\,\beta$, and $\gamma$ of the same integer $n$, are defined by
the identity
\begin{eqnarray}
  \label{eq:6}  [\alpha]\otimes [\beta]=\sum_{\gamma}k_{\alpha\,\beta\,\gamma}[\gamma].
\end{eqnarray}

The following classical result of Murnaghan and
Littlewood  (see \cite{Murna:kron}) shows that Kronecker coefficients
generalize Littlewood-Richardson coefficients.

\begin{coro}
  \begin{enumerate}
  \item If $k_{\alpha\,\beta\,\gamma}\neq 0$ then 
    \begin{eqnarray}
      \label{eq:7}
      (n-\alpha_1)+(n-\beta_1)\geq n-\gamma_1.
    \end{eqnarray}
\item Assume that equality holds in formula \eqref{eq:7} but
  not necessarily that $k_{\alpha\,\beta\,\gamma}\neq 0$.
Define $\bar\alpha=(\alpha_2\geq \alpha_3\cdots)$ and similarly
define $\bar \beta$ and $\bar\gamma$. Then
\begin{eqnarray}
  \label{eq:8}
  k_{\alpha\,\beta\,\gamma}=c_{\bar\alpha\,\bar\beta}^{\bar \gamma},
\end{eqnarray}
where $c_{\bar\alpha\,\bar\beta}^{\bar \gamma}$ is the
Littlewood-Richardson coefficient.
  \end{enumerate}
\end{coro}

\begin{proof}
Let us first introduce some notation on the linear group.
Let $V$ be a complex finite dimensional vector space and let  $\GL(V)$
be the corresponding linear group.
If $\alpha$ is a partition with at most dim$(V)$ parts, $S^\alpha V$
denotes the Schur power of $V$ ; it is the irreducible
$\GL(V)$-module of heighest weight $\alpha$. 
Let $\Fl(V)$ denote the variety of complete flags of $V$.
Given integers $a_i$ such that $1\leq a_1<\cdots<a_s\leq \dim (V)-1$,
 $\Fl(a_1,\cdots,a_s;\,V)$ denotes the set of flags
$V_1\subset\cdots\subset V_s\subset V$ such that $\dim(V_i)=a_i$ for
any $i$.

\bigskip
  Let us choose integers $e$ and $f$ such that 
 
\begin{eqnarray}
    \label{eq:9}\left\{
      \begin{array}{l}
    l(\alpha)\leq e,\\
l(\beta)\leq f,\\
l(\gamma)\leq e+f-1.     
      \end{array}
\right .
  \end{eqnarray}
Let $E$ and $F$ be two complex vector spaces of dimension $e$ and $f$.
Consider the group $G=\GL(E)\times\GL(F)$.
The Kronecker coefficient $k_{\alpha\,\beta\,\gamma}$ can be
interpreted in terms of representations of $G$. Namely  (see for
example \cite{Macdo,book:FH}) $k_{\alpha\,\beta\,\gamma}$ is the
multiplicity of $S^\alpha E\otimes S^\beta F$ in 
$S^\gamma(E\otimes F)$.
To interpret this multiplicity geometrically,  consider the variety
$$
X=\Fl(E)\times\Fl(F)\times\Fl(1,\cdots, e+f-1; E\otimes F)
$$
endowed with its natural $G$-action.
Consider the $\GL(E)$-linearized line bundle $\Li^\alpha$ on $\Fl(E)$
such that $H^0(\Fl(E),\Li^\alpha)=S^\alpha E$
(with usual notation, $\Li^\alpha=\Li_{-w_0\alpha}$).
Similarly, fix  $\Li^\beta$ on $\Fl(F)$ such that 
$H^0(\Fl(F),\Li^\beta)=S^\beta F$.
Because of assumption \eqref{eq:9}, there exists a 
$\GL(E\otimes F)$-linearized line bundle $\Li_\gamma$ on 
$\Fl(1,\cdots, e+f-1; E\otimes F)$ such that 
$H^0(\Fl(1,\cdots, e+f-1; E\otimes F),\Li^\gamma)=S^\gamma (E^*\otimes F^*)$. 
Observe that $S^\gamma (E^*\otimes F^*)$ is not a polynomial representation of $\GL(E)\times \GL(F)$. 
The line bundle
$\Li=\Li^\alpha\otimes\Li^\beta\otimes\Li_\gamma$ on $X$ is $G$-linearized.
Then 
\begin{eqnarray}
  \label{eq:10}
  k_{\alpha\,\beta\,\gamma}=\dim(H^0(X,\Li)^G).
\end{eqnarray}
Let $H_E,\,H_F,\,l_E$ and $l_F$ be hyperplanes and lines respectively
in $E$ and $F$ such that $E=H_E\oplus l_E$ and $F=H_F\oplus l_F$.
Let $\lambda$ be the one-parameter subgroup of $G$ acting on $H_E$ and $H_F$
with weight $1$ and on $l_E$ and $l_F$ with weight $0$.
Let $C_E$ be the set of complete flags of $E$ whose the hyperplane  is
$H_E$. Note that $C_E$ is an irreducible component of $\Fl(E)^\lambda$.
Similarly define $C_F$.
Let $C_{E\otimes F}$ be the set of points $V_1\subset \cdots\subset V_{e+f-1}$ in 
$\Fl(1,\cdots, e+f-1; E\otimes F)$ such that $V_1=l_E\otimes l_F$ and 
$V_{e+f-1}=(l_E\otimes l_F)\oplus (H_E\otimes l_F)\oplus(l_E\otimes H_F)$.
Note that $C_{E\otimes F}$ is an irreducible component of
$\Fl(1,\cdots, e+f-1; E\otimes F)^\lambda$ isomorphic to
$\Fl(H_E\oplus H_F)$. Then $C=C_E\times C_F\times 
C_{E\otimes  F}$ is an irreducible component of $X^\lambda$.

Observe that $C^+_{E\otimes F}$ is open in $\Fl(1,\cdots, e+f-1; E\otimes
F)$, $(C_E,\lambda)$ and $(C_F,\lambda)$ are  covering in $\Fl(E)$
and $\Fl(F)$ for the actions of $\GL(E)$ and $\GL(F)$. 
It follows that $(C,\lambda)$ is  covering. 

Let $x$ be a point in $C$. 
Let $\mu^\Li(x,\lambda)$ be the opposite of the weight of the action
of $\lambda$ on the fiber of $\Li$ over $x$.
\cite[Lemma 3]{GITEigen} implies that if $\dim(H^0(X,\Li)^G)>0$ then
$\mu^\Li(x,\lambda)\leq 0$ which is the inequality of the corollary.
Assume that $\mu^\Li(x,\lambda)= 0$, that is that $\lambda$
acts trivially on $\Li_{|C}$. Theorem \ref{th:mainintro}
shows that
$$
\dim(H^0(X,\Li)^G)=\dim(H^0(C,\Li_{|C})^{G^\lambda}).
$$
Moreover $\dim(H^0(C,\Li_{|C})^{G^\lambda})$ is the multiplicity of
the simple $\GL(H_E)\times\GL(H_F)$-module 
$S^{\bar\alpha}H_E\otimes S^{\bar\beta}H_F$ in the 
$\GL(H_E\oplus H_F)$-module
$S^{\bar \gamma}(H_E\oplus H_F)$.
By for example \cite[Chapter I, 5.9]{Macdo}, this multiplicity is
precisely $c_{\bar\alpha\,\bar\beta}^{\bar \gamma}$.
\end{proof}

\bibliographystyle{amsalpha}
\bibliography{branching}

\begin{center}
  -\hspace{1em}$\diamondsuit$\hspace{1em}-
\end{center}

\end{document}